\documentclass{amsart}

\usepackage{amssymb}
\usepackage{amscd}
\usepackage[all,cmtip]{xy}
\usepackage{hyperref}
\usepackage{graphicx}

\newcommand{\comment}[1]{}

\newcommand{\Z}{\mathbf{Z}}

\newcommand{\alg}{\mathrm{alg}}

\DeclareMathOperator{\Sym}{Sym}

\DeclareMathOperator{\Hom}{Hom}

\DeclareMathOperator{\Aut}{Aut}

\DeclareMathOperator{\Norm}{Norm}

\DeclareMathOperator{\Map}{Map}
\DeclareMathOperator{\Span}{Span}

\theoremstyle{plain}
\newtheorem{theorem}{Theorem}[section]

\newtheorem{lemma}[theorem]{Lemma}

\newtheorem{proposition}[theorem]{Proposition}

\theoremstyle{definition}

\newtheorem{remark}[theorem]{Remark}

\newtheorem{example}[theorem]{Example}

\begin{document}

\title{Polynomial maps on vector spaces over a finite field}
\author{Michiel Kosters}
\address{Mathematisch Instituut
P.O. Box 9512
2300 RA Leiden
the Netherlands}
\email{mkosters@math.leidenuniv.nl}
\urladdr{www.math.leidenuniv.nl/~mkosters}
\date{\today}
\thanks{This article is part of my PhD thesis written under the supervision of Hendrik Lenstra.}

\begin{abstract}
Let $l$ be a finite field of cardinality $q$ and let $n$ be in $\Z_{\geq 1}$. Let $f_1,\ldots,f_n \in l[x_1,\ldots,x_n]$ not all constant and
consider the evaluation map $f=(f_1,\ldots,f_n) \colon l^n
\to l^n$. Set $\deg(f)=\max_i \deg(f_i)$. Assume that $l^n
\setminus f(l^n)$ is not empty. We will prove
\begin{align*}
 |l^n\setminus f(l^n)| \geq \frac{n(q-1)}{\deg(f)}.
\end{align*}
This improves previous known bounds.
\end{abstract}

\maketitle

\section{Introduction}

The main result of \cite{WAN3} is the following theorem.

\begin{theorem} \label{27c13}
Let $l$ be a finite field of cardinality $q$ and let $n$ be in $\Z_{\geq 1}$. Let $f_1,\ldots,f_n \in l[x_1,\ldots,x_n]$ not all constant and
consider the map $f=(f_1,\ldots,f_n) \colon l^n
\to l^n$. Set $\deg(f)=\max_i \deg(f_i)$. Assume that $l^n
\setminus f(l^n)$ is not empty. Then we have
\begin{align*}
 |l^n\setminus f(l^n)| \geq \min\left\{\frac{n(q-1)}{\deg(f)},q\right\}.
\end{align*}
\end{theorem}
We refer to \cite{WAN3} for a nice introduction to this problem including references and historical remarks. The proof in \cite{WAN3} relies on $p$-adic liftings of such polynomial maps. We give a proof of a stronger statement using different techniques.

\begin{theorem} \label{27c14}
 Under the assumptions of Theorem \ref{27c13} we have
\begin{align*}
 |l^n\setminus f(l^n)| \geq \frac{n(q-1)}{\deg(f)}.
\end{align*}
\end{theorem}

We deduce the result from the case $n=1$ by putting a field structure $k$ on $l^n$ and relate the $k$-degree and the $l$-degree. We prove the result $n=1$ in a similar way as in \cite{TUR}.

\section{Degrees}

Let $l$ be a finite field of cardinality $q$ and let $V$ be a finite dimensional $l$-vector space. By $V^{\vee}=\Hom(V,l)$ we denote the dual of
$V$. Let $v_1,\ldots,v_f$ be a basis of $V$. By $x_1,\ldots,x_f$ we denote its dual basis in $V^{\vee}$, that is, $x_i$ is the map which sends $v_j$
to $\delta_{ij}$. Denote by $\Sym_l(V^{\vee})$ the symmetric algebra
of $V^{\vee}$ over $l$. We have an isomorphism $l[x_1,\ldots,x_f] \to \Sym_l(V^{\vee})$ mapping
$x_i$ to $x_i$. Note that $\Map(V,l)=l^V$ is a commutative ring under the coordinate wise addition and multiplication and it is a $l$-algebra. The
linear
map $V^{\vee} \to
\Map(V,l)$ induces by the universal property of $\Sym_l(V^{\vee})$ a ring morphism $\varphi \colon \Sym_l(V^{\vee}) \to \Map(V,l)$. When choosing a basis,
we have the following commutative diagram, where the second horizontal map is the evaluation map and the vertical maps are the natural isomorphisms:

\[
\xymatrix{
\Sym(V^{\vee}) \ar[r]^{\varphi} & \Map(V,l) \\
l[x_1,\ldots,x_n] \ar[u] \ar[r] & \Map(l^n,l). \ar[u]
}  
\]

\begin{lemma} \label{27c71}
 The map $\varphi$ is surjective. After a choice of a basis as above the kernel is equal to $(x_i^q-x_i: i=1,\ldots,f)$ and every $f \in \Map(V,l)$
has a unique representative $\sum_{m=(m_1,\ldots,m_n):\ 0 \leq m_i \leq q-1} c_m x_1^{m_1}\cdots x_n^{m_n}$ with $c_m \in l$.
\end{lemma}
\begin{proof}
 After choosing a basis, we just consider the map 
\begin{align*} 
 l[x_1,\ldots,x_f] \to \Map(l^n,l).
\end{align*} 
For $c=(c_1,\ldots,c_f) \in l^f$ set
\begin{align*}
 f_c = \prod_{i} (1-(x_i-c_i)^q).
\end{align*}
For $c' \in l^f$ we have $f_c(c')=\delta_{cc'}$. With these building blocks one easily shows that $\varphi$ is surjective.

For $i \in \{1,2,\ldots,f\}$ the element $x_i^q-x_i$ is in the kernel of $\varphi$. This shows that modulo the kernel any $g \in l[x_1,\ldots,x_f]$
has a representative 
\begin{align*}
 f = \sum_{m=(m_1,\ldots,m_r):\ 0 \leq m_i \leq q-1} c_m x_1^{m_1}\cdots x_r^{m_n}.
\end{align*}
The set of such elements has cardinality
$q^{q^r}$. As $\#
\Map(V,l)=q^{q^r}$, we see that the kernel is $(x_i^q-x_i: i=1,\ldots,r)$. Furthermore, any element has a unique representative as described above.
\end{proof}

Note that $\Sym_l(V^{\vee})$ is a graded $l$-algebra where we say that $0$ has degree $-\infty$. For $f \in \Map(V,l)$ we set 
\begin{align*}
\deg_l(f)=\mathrm{min}\left(\deg(g):\varphi(g)=f\right). 
\end{align*}
Note that $\deg_l(f_1+f_2) \leq \max(\deg_l(f_1),\deg_l(f_2))$, with equality if the degrees are different. In practice, if $f \in l[x_1,\ldots,x_n]$,
then $\deg_l(f)$ is calculated as follows: for all $i$ replace $x_i^q$ by $x_i$ until $\deg_{x_i}(f)<q$. Then the degree is the total degree of
the remaining polynomial.

Let $W$ be a finite dimensional $l$-vector space. Then we have $\Map(V,W)=W \otimes_l \Map(V,l)$. For $f \in \Map(V,W)$ we set
\begin{align*}
 \deg_l(f) = \max \left( \deg_l(g \circ f): g \in W^{\vee} \right).
\end{align*}
If $g_1,\ldots,g_n$ is a basis of $W^{\vee}$, then $\deg_l(f)= \max \left( \deg_l(g_i \circ f): i =1,\ldots,n \right)$. This follows from the
identity $\deg_l( \sum_{i} c_i g_i) \leq \max( \deg_l(g_i))$ for $c_i \in l$. Note that the degree is bounded above by $(q-1) \cdot \dim_l(V)$. 

For $i \in \Z_{\geq 0}$ and a subset $S$ of $\Sym_l(V^{\vee})$ we set 
\begin{align*}
 S^i_l = \Span_l( s_1 \cdot \ldots \cdot s_i: s_i \in S) \in \Sym_l(V^{\vee}).
\end{align*}

\begin{lemma} \label{27c90}
Let $f \in \Map(V,W)$. 
For $i \in \Z_{\geq 0}$ one has: $\deg_f \leq i$ $\iff$ $f \in W \otimes_l (l+V^{\vee})^i_l$.
\end{lemma}
\begin{proof}
Suppose first that $W=l$. The proof comes down to showing the following identity for $i \in \Z_{\geq 0}$:
\begin{align*}
 l+V^{\vee}+\ldots+ \left(V^{\vee} \right)^i_l = (l+V^{\vee})^i_l.
\end{align*}
The general case follows easily.
\end{proof}

\section{Relations between degrees}
Let $k$ be a finite field and let $l$ be a subfield of cardinality $q$. Set $h=[k:l]$. Let $V$ and $W$ be finite dimension $k$-vector spaces. Let $f
\in \Map(V,W)$. In this section we will describe the relation between the $k$-degree and the $l$-degree.

Let us first assume that $W=k$. Let $v_1,\ldots,v_r$ be a basis of $V$ over $k$. Let $R=k[x_1,\ldots,x_r]/(x_1^{q^h}-x_1,\ldots,x_r^{q^h}-x_r)$. We
have the following diagram where all morphisms are ring morphisms. Here $\psi$ is the map discussed before, $\tau$ is the natural isomorphism,
$\overline{\varphi}$ is the isomorphism discussed before, and $\sigma$ is the isomorphism, depending on the basis, discussed above.
\[
\xymatrix{
k \otimes_l \Map(V,l) \ar[r]^{\tau} & \Map(V,k) &\\
k \otimes_l \Sym_l(\Hom_l(V,l)) \ar[u]^{\psi} & \Sym_k(\Hom_k(V,k))/\ker(\varphi) \ar[u]_{\overline{\varphi}} \ar[r]^{\ \ \ \ \ \ \ \ \ \ \ \ \sigma}
& R.
}  
\]
Consider the ring morphism $\rho= \sigma \circ \overline{\varphi}^{-1} \circ \tau \circ \psi \colon k \otimes_l \Sym_l(\Hom_l(V,l)) \to R$.
Lemma \ref{27c90} suggest that to compare degrees, we need to find 
\begin{align*}
\rho(k \otimes_l
(l+\Hom_l(V,l))^i_l). 
\end{align*}
The following lemma says that it is enough to find $k+k \otimes_l \Hom_l(V,l)$.

\begin{lemma} \label{27c91}
For $i \in \Z_{\geq 0}$ we have the following equality in $k \otimes_l \Sym_l(V)$:
\begin{align*}
 k \otimes_l (l + \Hom_l(V,l))^i_l = \left(k+k \otimes_l \Hom_l(V,l) \right)^i_k.
\end{align*} 
\end{lemma}
\begin{proof}
 Both are $k$-vector spaces and the inclusions are not hard to see.
\end{proof}

The following lemma identifies $k+k \otimes_l \Hom_l(V,l)$.

\begin{lemma} \label{27c92}
One has
\begin{align*}
\rho(k+k \otimes_l \Hom_l(V,l)) = \Span_k \left(\{ x_j^{q^s}: 1 \leq j \leq r,\ 0 \leq s <h \} \sqcup
\{1\} \right).
\end{align*}
\end{lemma}
\begin{proof}
Note that $\tau \circ \phi (k+k \otimes_l \Hom_l(V,l))= k \oplus \Hom_l(V,k) \subseteq \Map(V,k)$. Note that 
\begin{align*}
\sigma^{-1} \left(\Span_k \left(\{ x_j^{q^s}: 1 \leq j \leq r,\ 0 \leq s <h \} \right) \right)\subseteq \Hom_l(V,k) .
\end{align*}
As both sets have dimension $\dim_l(V)=r \cdot h$ over $k$, the result follows.
\end{proof}

For $m, n \in \Z_{\geq 1}$ we set $s_m(n)$ to be the sum of the digits of $n$ in base $m$.

\begin{lemma} \label{27c93}
Let $m \in \Z_{\geq 2}$ and $n, n' \in \Z_{\geq 0}$. 
Then the following hold:
\begin{enumerate}
 \item $s_m(n+n') \leq s_m(n)+s_m(n')$; 
 \item Suppose $n=\sum_{i} c_i m^i$, $c_i \geq 0$. Then we have $\sum_i c_i \geq s_m(n)$ with equality iff for all $i$ we have $c_i<m$.
\end{enumerate}
\end{lemma}
\begin{proof}
i. This is well-known and left to the reader.

ii. We give a proof by induction on $n$. For $n=0$ the result is correct. Suppose first that $n=c_s m^s$ and assume that $c_s \geq m$. Then we have
$n=(c_s-m)m^s+m^{s+1}$. By induction and i we have 
\begin{align*}
 c_s > c_s-m + 1 \geq s_m((c_s-m)m^s)+s_m(m^{s+1}) \geq s_m(c_s m^s).
\end{align*}

In general, using i, we find
\begin{align*}
 \sum_i c_i \geq \sum_i s_m(c_i m^i) \geq s_m(n).
\end{align*}
Also, one easily sees that one has equality iff all $c_i$ are smaller than $m$.
\end{proof}

\begin{proposition} \label{27c33}
Let $f \in k[x_1,\ldots,x_r]$ nonzero with the degree in all
$x_i$ of all the monomials less than $q^h$. Write $f= \sum_{s=(s_1,\ldots,s_r)} c_s x_1^{s_1}\cdots x_r^{s_r}$. Then the $l$-degree of
$\tau^{-1} \circ \overline{\varphi}(\overline{f}) \in k \otimes_l \Map(V,l)$ is equal to
\begin{align*}
\max\{s_q(s_1)+\ldots+s_q(s_r): s=(s_1,\ldots,s_r) \textrm{ s.t. } c_s \neq 0\}.
\end{align*}
\end{proposition}
\begin{proof}
Put $g=\tau^{-1} \circ \overline{\varphi} \circ \sigma^{-1} (\overline{f})$. From Lemma \ref{27c90}, Lemma \ref{27c91} and
Lemma \ref{27c92} we obtain the following. Let $i \in \Z_{\geq 0}$. Then $\deg_l(g) \leq i$ iff 
\begin{align*}
g \in \rho(k \otimes_l (l + \Hom_l(V,l))^i_l) &= \rho (\left(k+k \otimes_l \Hom_l(V,l) \right)^i_k) \\
&= \left(\Span_k \left(\{ x_j^{q^s}: 1 \leq j \leq r,\ 0 \leq s <h \} \sqcup \{1\} \right) \right)_k^i.
\end{align*}
The result follows from Lemma \ref{27c93}.
\end{proof}

The case for a general $W$ just follows by decomposing $W$ into a direct sum of copies of $k$ and then taking the maximum of the corresponding
degrees.

\section{Proof of main theorem}

\begin{lemma} \label{27c67}
 Let $m , q, h \in \Z_{>0}$ and suppose that $q^h-1|m$. Then we have: $s_q(m) \geq h(q-1)$.
\end{lemma}
\begin{proof}
We do a proof by induction on $m$.

 Suppose that $m<q^h$. Then $m=q^h-1$ and we have $s_q(m)=h(q-1)$.
 
 Suppose $m \geq q^h$. Write $m=m_0q^h+m_1$ with $0 \leq m_1<q^h$ and $m_0 \geq 1$. We claim that $q^h-1|m_0+m_1$. Note that $
m_0+m_1 \equiv m_0q^h+m_1 \equiv 0 \pmod{q^h-1}$. 
Then by induction we find
\begin{align*}
 s_q(m)=s_q(m_0)+s_q(m_1) \geq s_q(m_0+m_1) \geq h(q-1).
\end{align*}
\end{proof}

\begin{lemma} \label{27c89}
Let $k$ be a finite field of cardinality $q'$. Let $R=k[X_{a}: a \in k]$ and consider the action of
$k^{*}$ on $R$ given by
\begin{align*}
 k^* &\mapsto \Aut_{k-\alg}(R) \\
  c &\mapsto (X_a \mapsto X_{ca}).
\end{align*}
Let $F \in R$ fixed by the action of $k^*$ with $F(0,\ldots,0)=0$ and such that the degree of no monomial of $F$ is a multiple of $q'-1$. Then
for $w=(a)_{a} \in k^k$ we
have $F(w)=0$.
\end{lemma}
\begin{proof}
 We may assume that $F$ is homogeneous with $d=\deg(F)$ which is not a multiple of $q'-1$. Take $\lambda \in k^*$ a generator of the cyclic group.
As $F$ is fixed by $k^*$ we find: 
\begin{align*}
 F(w) = F( \lambda w)=\lambda^{d} F(w).
\end{align*}
As $\lambda^d \neq 1$, we have $F(w)=0$ and the result follows.
\end{proof}

Finally we can state and prove a stronger version of Theorem \ref{27c14}.

\begin{theorem} \label{27c18}
Let $k$ be a finite field. Let $l \subseteq k$ be a subfield with $[k:l]=h$ and let $V$ be a finite dimensional $k$-vector space. Let $f \in
\Map(V,V)$ be a non-constant and non-surjective map. Then $f$ misses at least
\begin{align*}
 \frac{\dim_k(V) \cdot h \cdot (\#l-1)}{\deg_l(f)}
\end{align*}
values.
\end{theorem}
\begin{proof}
Set $\#l =q$. Put a $k$-linear multiplication on $V$ such that it becomes a field. This allows us reduce to the case where $V=k$. Assume $V=k$.
After shifting we may assume $f(0)=0$. Put an ordering $\leq$ on $k$. 
In $k[T]$ we have
\begin{align*}
 \prod_{a \in k} (1-f(a)T) = 1 - \sum_{a} f(a)T+ \sum_{a<b} f(a)f(b)T^2 - \ldots = \sum_{i} a_i T^i.
\end{align*}
For $1 \leq i < \frac{h(q-1)}{\deg_l(f)}$ we claim that $a_i=0$. Put $f_0 \in k[x]$ a polynomial of degree at most $q^h-1$ inducing $f \colon k \to k$. 
Consider $g_i=\prod_{a_1<\ldots<a_i} f_0(X_{a_1})\cdots f_0(X_{a_i})$ in $k[X_{a}: a \in k]$, which is fixed by $k^*$. We have a map
\begin{align*}
\varphi \colon k[X_a: a \in k] \to \Map(k^k,k).
\end{align*}
 Proposition \ref{27c33} gives us that $\deg_l(\varphi(g_i)) = i \deg_l(f)<h (q-1)$.
We claim that there is no monomial in $g_i$ with degree a multiple of $q^h-1$. Indeed, suppose that there is a monomial $c X_{a_1}^{r_1}\cdots
X_{a_i}^{r_i}$ ($c \neq 0$) in $g_i$ (note that not all $r_i$ are zero) and suppose that $q^h-1|\sum_i r_i$. Then by Lemma \ref{27c67} and Proposition
\ref{27c33} we have
\begin{align*}
 h(q-1) \leq s_q(\sum_j r_j) \leq  \sum_j s_q(r_j) \leq i \cdot \deg_l(f) < h(q-1),
\end{align*}
contradiction.
Hence we can apply Lemma \ref{27c89} to conclude that $a_i=0$. 

Hence we conclude 
\begin{align*}
\prod_{a \in k} (1-f(a)T) &\equiv 1 \pmod{T^{\lceil \frac{h(q-1)}{\deg_l(f)} \rceil}}.
\end{align*}
Similarly, for the identity function of $l$-degree $1$, we conclude
\begin{align*}
\prod_{a \in k} (1-aT) &\equiv 1 \pmod{T^{\lceil \frac{h(q-1)}{\deg_l(f)} \rceil}}.
\end{align*}
Combining this gives:
\begin{align*}
 \prod_{a \in k \setminus f(k)} (1-aT) &= \frac{\prod_{a \in k}(1-aT)}{\prod_{b \in f(k)}(1-bT)} \\
&\equiv \frac{\prod_{a \in k}(1-aT)}{\prod_{b \in f(k)}(1-bT)} \cdot \prod_{c \in k} (1-f(c)T) \pmod{T^{\lceil \frac{h(q-1)}{\deg_l(f)}
\rceil}}\\
&\equiv \prod_{b \in f(k)} (1-bT)^{\#f^{-1}(b)-1} \pmod{T^{\lceil \frac{h(q-1)}{\deg_l(f)} \rceil}}.
\end{align*}
Note that the polynomials $\prod_{a \in k \setminus f(k)} (1-aT)$ and $\prod_{b \in f(k)} (1-bT)^{\#f^{-1}(b)-1}$ have degree bounded by $s=k
\setminus f(k)$ and are
different since $s \geq 1$. But this implies that $s \geq \frac{h(q-1)}{\deg_l(f)}$.
\end{proof}

\begin{remark}
Different $l$ in Theorem \ref{27c18} may give different lower bounds.
\end{remark}

\section{Examples}

In this section we will give examples which meet the bound from Theorem \ref{27c14}.

\begin{example}[$n=\deg(f)$]
Let $l$ be a finite field of cardinality $q$. In this example we will show that for $n, d \in \Z_{\geq 2}$ there are functions $f_1,\ldots,f_n \in
l[x_1,\ldots,x_n]$
such that the maximum of the degrees is equal to $d$ such that the induced map $f \colon l^n \to l^n$ satisfies $|l^n \setminus f(l^n)|=\frac{n(q-1)}{d}=q-1$.
For $i=1,\ldots,n-1$ set $f_i=x_i$. Let $l_{n-1}$ be the unique extension of $l$ of degree $n-1$. Let $v_1,\ldots,v_{n-1}$ be a basis of $l_{n-1}$ over $l$.
Then $g=\Norm_{l_{n-1}/l}(x_1v_1+\ldots+x_{n-1} v_{n-1})$ is a homogeneous polynomial of degree $n-1$ in $x_1,\ldots,x_{n-1}$. Put $f_n=x_n \cdot g$.
As the norm of a nonzero element is nonzero, one easily sees that $l^n \setminus f(l^n)=\{0\} \times \ldots \{0\} \times l^*$ has cardinality $q-1$. 
\end{example}

\begin{example}[$n=\frac{\deg(f)}{q-1}$]
 Let $l$ be a finite field and let $n \in \Z_{\geq 1}$. Let $f_1,\ldots,f_n \in l[x_1,\ldots,x_n]$ such that the combined map $f \colon
l^n \to l^n$ satisfies $|l^n \setminus f(l^n)|=1$ (Lemma \ref{27c71}). From Theorem \ref{27c14} and the upper bound $n(q-1)$ for the degree we deduce that $\deg(f)=n(q-1)$.
\end{example}

\end{document}